\documentclass{amsart}
\usepackage{amsmath, amsthm, amssymb,latexsym,enumerate,mathrsfs}
\usepackage[all]{xy}
\usepackage{enumitem}
\usepackage{hyperref}
\usepackage{invariant}
\usepackage{xcolor}
\usepackage{tikz-cd}
\usepackage{verbatim}

\title[Algebraic maps constant on unpolarized isomorphism classes]{Algebraic maps constant on isomorphism classes of unpolarized abelian varieties are constant}

\author[E.\ Rains]{E.\ Rains}
\address{Department of Mathematics, California Institute of Technology, Pasadena, CA, USA}
\email{rains@caltech.edu}
\author[K.\ Rubin]{K.\ Rubin}
\address{Department of Mathematics, University of California, Irvine, CA 92697, USA}
\email{krubin@uci.edu}
\author[T.\ Scholl]{T.\ Scholl}
\address{Department of Mathematics, University of California, Irvine, CA 92697, USA}
\email{schollt@uci.edu}
\author[S.\ Sharif]{S.\ Sharif}
\address{Department of Mathematics, California State University San Marcos, San Marcos, CA 92096, USA}
\email{ssharif@csusm.edu}
\author[A.\ Silverberg]{A.\ Silverberg}
\address{Department of Mathematics, University of California, Irvine, CA 92697, USA}
\email{asilverb@uci.edu}
\thanks{We thank the Alfred P.~Sloan Foundation and the National Science Foundation, which provided support for the research, and the June 2019 ICERM workshop on Arithmetic of Low-Dimensional Abelian Varieties (funded by the Simons Collaboration on Arithmetic Geometry, Number Theory, and Computation), where this project began.}

\begin{document}

\begin{abstract}
We show that if a 
rational map is constant on each isomorphism class of unpolarized abelian varieties of a given dimension, then it is a constant map.
Our results shed light on a question raised in \cite{multiparty} concerning a proposal for multiparty non-interactive key exchange.
\end{abstract}

\maketitle


\section{Introduction}
\label{sec:introduction}

If $A_1$ and $A_2$ are abelian varieties, we say $A_1$ and $A_2$ are \emph{weakly isomorphic}, written $A_1 \approx A_2$, if $A_1$ and $A_2$ are isomorphic as unpolarized abelian varieties. Let $\ag$ denote the coarse moduli scheme of principally polarized abelian varieties of dimension $g$. One of our main results is the following.
\begin{theorem}\label{thm:invariant-c-constant}
  Suppose that $g\in\Z_{\ge 2}$, that $k$ is an algebraically closed field, and that $f: \ag \dashrightarrow X$ is a rational map of $k$-schemes. Suppose that $f(A_1,\lambda_1) = f(A_2,\lambda_2)$ whenever $A_1 \approx A_2$ and $(A_1,\lambda_1)$ and $(A_2,\lambda_2)$ are in the domain of definition of $f$. Then $f$ is a constant function.
\end{theorem}
Note that Theorem~\ref{thm:invariant-c-constant} fails when $g = 1$; the $j$-invariant gives a nonconstant morphism $\mathcal{A}_1 \to \Pro^1$ such that $j(E_1) = j(E_2)$ whenever $E_1 \approx E_2$.



If $R$ is a scheme, let
$$
\sg(R) = \{(A_1,A_2) \in (\ag \times \ag)(R) | A_1 \approx A_2\}.
$$
We will deduce Theorem \ref{thm:invariant-c-constant} from the following result.

\begin{theorem}\label{thm:sg-c-dense}
  If $g\in\Z_{\ge 2}$ and $k$ is an algebraically closed field, then the set $\sg(k)$ is Zariski dense in the $k$-scheme $\ag \times \ag$.
\end{theorem}

It is widely acknowledged that in order to obtain a moduli space of abelian varieties, one must first fix a polarization.  However, the mere fact that isomorphism as unpolarized abelian varieties itself is not a well behaved equivalence relation does not {\it a priori} exclude the possibility that some slightly coarser equivalence relation {\em could} give rise to a moduli space. Theorem \ref{thm:sg-c-dense} implies that this cannot happen: the only coarser equivalence relation that is also Zariski closed is the trivial relation in which all elements are equivalent.
(Since the Zariski closure of an equivalence relation need not be an equivalence relation, Theorem \ref{thm:sg-c-dense} is somewhat stronger than this.) Furthermore, all of the results stated in this section hold even when $k$ is replaced by a geometrically reduced scheme; see Remark~\ref{rmk:replace-with-scheme}.
Our results can be viewed as making precise the discussion in~\cite[p. 97]{git} about  ``pathologies'' that arise when considering the relation $\approx$.

We will prove the following result in \S\ref{sec:characteristic-p}.
\begin{theorem}\label{thm:arbit-char}
  Suppose that $g\in\Z_{\ge 2}$, $k$ is an algebraically closed field, and that $f: \ao^g \dashrightarrow X$ is a rational map of $k$-schemes. Suppose that $f(A_1,\lambda_1) = f(A_2,\lambda_2)$ whenever $A_1 \approx A_2$ and $(A_1,\lambda_1)$ and $(A_2,\lambda_2)$ are in the domain of definition of $f$. Then $f$ is a constant function.
\end{theorem}


A motivation behind Theorem \ref{thm:arbit-char} is the proposal for a construction of a cryptographic protocol in~\cite{multiparty}. In that protocol, $n$ parties each construct a product of elliptic curves over a finite field such that any pair of products is isomorphic. Put another way, each party computes a point of $\ao^g$ so that the chosen points are weakly isomorphic to each other. However, the proposal for a protocol was incomplete. The open question in~\cite{multiparty} was whether one can extract a numerical invariant of the product of elliptic curves that respects weak isomorphism, that is, a non-constant map $f: \ao^g \to X$ for a suitable space $X$ such that $f(A_1) = f(A_2)$ whenever $A_1 \approx A_2$. In this context, Theorem~\ref{thm:arbit-char} shows that if $f$ is algebraic, then it is constant, and thus not useful cryptographically. It remains an open problem whether there is a useful non-algebraic invariant of (isomorphism classes of non-polarized) products of elliptic curves.

Since $\ag$ is only a coarse moduli space, one might wonder whether Theorem~\ref{thm:invariant-c-constant} still holds if $\ag$  is replaced by the fine moduli space of abelian varieties with full level $n$ structure with $n\ge 3$ and not divisible by $\car(k)$. The following generalization of
Theorem~\ref{thm:invariant-c-constant} shows that the phenomenon persists even for the fine moduli space.

\begin{corollary}\label{cor:agn}
  Let $g,n \in \Z_{\geq 2}$, and let $\agn$ be the moduli space of principally polarized abelian varieties with full level $n$ structure. Suppose that $k$ is an algebraically closed field with $\car(k) \nmid n$,  and $f: \agn \dashrightarrow X$ is a rational map of $k$-schemes. View points of $\agn$ as triples  $(A, \lambda,L)$, where $A$ is an abelian variety with principal polarization $\lambda$, and $L$ is a symplectic basis for the $n$-torsion $A[n]$ (where symplectic means with respect to the Weil pairing induced by $\lambda$). Suppose $f(A_1,\lambda_1, L_1) = f(A_2,\lambda_2, L_2)$ whenever $A_1 \approx A_2$ and $(A_1,\lambda_1, L_1)$ and $(A_2,\lambda_2, L_2)$ are in the domain of definition of $f$. Then $f$ is a constant function.
\end{corollary}


Suppose $k$ is an algebraically closed field. If $\mu: A \to A^\vee$ is a polarization on a $g$-dimensional abelian variety $A$ over $k$ of degree not divisible by $\car(k)$, then there are unique positive integers $d_1 \mid \cdots \mid d_g$ such that
\[
  \ker \mu \cong (\Z/d_1\Z \times \cdots \times \Z/d_g\Z)^2
\]
and $\car(k) \nmid d_g$ (see \cite[Theorem 1 et seq.]{mumford1966}). Letting $D = (d_1, \ldots, d_g)$, we say that the polarization $\mu$ has \emph{type} $D$.
A \emph{polarization type} is a tuple ${D} = (d_1, \ldots, d_g)$ where the $d_i$ are positive integers such that $d_1 \mid \cdots \mid d_g$. If $D$ is a polarization type, let $\agd$ denote the moduli space of abelian varieties with a polarization of type ${D}$. View points of $\agd$ as pairs $(A,\mu)$, where $A$ is an abelian variety and $\mu$ is a polarization of type ${D}$.

The following result, which we prove in \S\ref{sec:Cor15proof} below, is a generalization of Theorem \ref{thm:invariant-c-constant} with $\ag$ replaced by the moduli space of  abelian varieties with a fixed 
polarization type.

\begin{corollary}\label{cor:agpol}
  Suppose that $g \in \Z_{\geq 2}$, ${D} = (d_1, \ldots, d_g)$ is a polarization type, $k$ is an algebraically closed field, $\car(k) \nmid d_g$, and  $f: \agd \dashrightarrow X$ is a rational map of $k$-schemes. Suppose $f(A_1, \mu_1) = f(A_2, \mu_2)$ whenever $A_1 \approx A_2$ and $(A_1,\mu_1)$ and $(A_2,\mu_2)$ are in the domain of definition of $f$. Then $f$ is a constant function.
\end{corollary}

\begin{corollary}
\label{cor:alldcor}
  Suppose that $g\in\Z_{\ge 2}$, that $k$ is an algebraically closed field, and that $f: \cup_{D} \agd \to X$ is a morphism of $k$-schemes, where $D$ runs over all polarization types $(d_1, \ldots, d_g)$ for which $\car(k) \nmid d_g$. Suppose that $f(A_1,\lambda_1) = f(A_2,\lambda_2)$ whenever $A_1 \weak A_2$. Then $f$ is a constant function.
\end{corollary}



We next discuss the proofs.
If $N$ is a positive integer, let $Y_0(N)/\Spec(\Z[1/N])$ denote the modular curve parametrizing pairs $(E, C)$, where $E$ is an elliptic curve and $C \subset E$ is a cyclic subgroup of $E$ of order $N$. If $A$ is a positive definite, symmetric, integer $g\times g$ matrix, define a map $\psimod: Y_0(\det(A)) \to \ag$ as follows. Suppose $(E, C) \in Y_0(\det(A))$. The matrix $A$ induces a natural endomorphism $\lambda_A$ of $E^g$. Since $A$ is symmetric and positive definite, we can view $\lambda_A$ as a polarization on $E^g$. Let $B = E^g/((\ker \lambda_A) \cap C^g)$ and let $\pi: E^g \to B$ be the quotient map. Then by~\cite[Prop.~16.8]{milne-av} there is a unique polarization $\lambda$ on $B$ such that $\pi^*(\lambda) = \lambda_A$. As we shall see in Lemma~\ref{lem:psimod-weakly-isomorphic-to-product}, $\lambda$ is a principal polarization, and $B$ is weakly isomorphic to a product of elliptic curves.  Define
\begin{equation}\label{def:psimoddef}
\psimod: Y_0(\det(A)) \to \ag  \text{ by } \psimod(E,C) = (B,\lambda).
\end{equation}
For another viewpoint on the maps $\psimod$, see \cite[p. 19 et seq.]{rains}. In Prop.~\ref{prop:psimod-over-c} below we obtain a more concrete characterization of $\psimod$ in the case where the ground field is the field of complex numbers. 

Let
$$X_A=\psimod(Y_0(\det(A))).$$

\begin{definition}\label{def:detl}
If $\ell$ is a prime number and $g$ is a positive integer, let $\detl$ denote the set of positive definite symmetric $g \times g$ integer matrices whose determinant is a power of $\ell$.
\end{definition}

\begin{theorem}\label{thm:Sg-dense}
  Suppose $g\in\Z_{\ge 2}$ and $k$ is an algebraically closed field.
  Then there are infinitely many prime numbers $\ell$ such that if
  $A,A' \in \detl$, then $\sg(k) \cap (X_A \times X_{A'})(k)$ is Zariski dense in $X_A \times X_{A'}$.
\end{theorem}

\begin{theorem}\label{thm:curves-dense}
    Suppose $\ell$ is a prime number, $g\in\Z_{\ge 1}$, and $k$ is an algebraically closed field with $\car(k) \neq \ell$. Then as $k$-schemes, $\bigcup_{A \in \detl} X_A$ is Zariski dense in $\ag$.
\end{theorem}

Theorems \ref{thm:Sg-dense} and \ref{thm:curves-dense} are proved in \S\ref{sec:step-2} and \S\ref{sec:step-1}, respectively. Proposition~\ref{prop:lim-degree} is the only place where we require that $g > 1$.

Next we derive Theorem \ref{thm:sg-c-dense} from Theorems \ref{thm:Sg-dense} and \ref{thm:curves-dense},  we derive Theorem \ref{thm:invariant-c-constant} from Theorem \ref{thm:sg-c-dense},  we derive Corollary \ref{cor:agn} from Theorem \ref{thm:invariant-c-constant}, and we derive Corollary \ref{cor:alldcor} from Corollary~\ref{cor:agpol}.

\begin{proof}[Proof of Theorem~\ref{thm:sg-c-dense}]
Fix a prime number $\ell\neq\car(k)$ that satisfies the conclusion of Theorem \ref{thm:Sg-dense}. We have
\begin{multline*}
\bigcup_{A,A' \in \detl} \left(\sg(k) \cap (X_A \times X_{A'})(k)\right) \\
 =  \sg(k) \cap \left(\bigcup_{A,A' \in \detl} (X_A \times X_{A'})(k)\right)
   \subset \sg(k).
\end{multline*}
By Theorem~\ref{thm:Sg-dense}, the set
$\bigcup_{A,A' \in \detl} \left(\sg(k) \cap (X_A \times X_{A'})(k)\right)$ is Zariski dense in
$$
\bigcup_{A,A' \in \detl} \left(X_A \times X_{A'}\right) =
    \left(\bigcup_{A \in \detl} X_A\right) \times \left(\bigcup_{A' \in \detl} X_{A'}\right),
    $$
    which by Theorem~\ref{thm:curves-dense} is Zariski dense in $\ag \times \ag$.
\end{proof}

\begin{proof}[Proof of Theorem~\ref{thm:invariant-c-constant}]
  Let $U$ be the domain of definition of $f$, so $f: U \to X$ is a morphism. Consider the fiber product $\df = U \times_X U$.
  The universal property of $\df$ says that for any scheme $W$ and morphisms $h_1$ and $h_2$ from $W$ to $U$ such that $f \circ h_1 = f \circ h_2$, there is a unique morphism $h: W \to \df$ such that $h_1$ and $h_2$ factor through $h$.

  Applying the universal property with $W = \Spec k$ shows that $\df(k)$ consists of the pairs $(A_1,A_2) \in U \times U$ of abelian varieties over $k$ such that $f(A_1) = f(A_2)$.
  By the hypothesis on $f$, if $A_1 \approx A_2$ then $f(A_1) = f(A_2)$. Hence 
  $$(\sg \cap (U \times U))(k) \subseteq \df(k).$$
  Since $\df$ is a closed subscheme of $U \times U$, Theorem~\ref{thm:sg-c-dense} implies that $\df = U \times U$, as desired.
\end{proof}

\begin{proof}[Proof of Corollary~\ref{cor:agn}]
  Let $\pi: \agn \to \ag$ be the canonical map that ``forgets'' the level $n$ structure. The space $\agn$ is Galois over $\ag$; let $G$ denote the Galois group, so 
  $G = \Sp_{2g}(\Z/n\Z)$. Letting $X^G$ denote the product of $\# G$ copies of $X$, indexed by $G$, then the group $G$ acts on $X^G$ by permuting the factors: for $x \in X^G$ and $h,g \in G$, if the $g$th coordinate of $x$ is $x_g$, then the $g$th coordinate of $h(x)$ is $x_{h^{-1}g}$.
  Let $X^G/G$ denote the quotient of $X^G$ by this action.

  Let $V$ denote the domain of definition of $f$, and let $U = \pi(V)$, which is an open subset of $\ag$. Define a map $f_G: U \to X^G/G$ as follows. For $u \in U$, choose  $v \in \pi^{-1}(x)$. Then $\pi^{-1}(u) = \{g(v)\}_{g \in G} \subset \agn$. Let $z = (f(g(v)))_{g \in G} \in X^G$, and let $f_G(u)$ be the image of $z$ in $X^G/G$.

  Since $\pi(A,L) = A$, we have $f_G(A_1) = f_G(A_2)$ whenever $A_1 \approx A_2$ with $A_1,A_2 \in U$. By Theorem~\ref{thm:invariant-c-constant}, $f_G$ is constant. Since $\agn$ is connected, it follows that $f$ is constant.
\end{proof}

\begin{proof}[Proof of Corollary~\ref{cor:alldcor}]
  By Corollary~\ref{cor:agpol}, $f$ is constant on each $\agd$.
 Fix two polarization types $D$ and $D'$. Our goal is to show that $f(\agd) = f(\agdp)$. Choose an elliptic curve $E$ over $k$. Then $E^g$ has polarizations $\mu$ and $\mu'$ over $k$ of types $D$ and $D'$, respectively, so $(E^g, \mu) \in \agd(k)$ and $(E^g,\mu') \in \agdp(k)$. Since $E^g \weak E^g$, by our hypothesis we have $f(E^g, \mu) = f(E^g, \mu')$. Thus $f(\agd) = f(\agdp)$, as desired.
\end{proof}

\begin{remark}\label{rmk:replace-with-scheme}
  All of the above results 
hold
if we replace the algebraically closed field $k$ with a connected, geometrically reduced scheme $S$; for example, if we replace $k$ with (Spec of) an integral domain. To see this, note that when $S$ is geometrically integral, we can use the fact that the generic point is dense in $S$. For $S$ geometrically reduced, instead consider the union of the generic points of the irreducible components. To generalize our results that assume non-divisibility by the characteristic of $k$ (for example, $\car(k) \nmid n$), we replace this hypothesis with the assumption that the relevant value 
is relatively prime to the characteristics of all geometric generic points. 
\end{remark}

\section{Proof of Theorem~\ref{thm:Sg-dense}}
\label{sec:step-2}


The plan is to find a suitable infinite set of pairs $(x,y) \in \sg(k) \cap (X_A \times X_{A'})(k)$. We will explicitly construct these pairs by choosing $x,y \in \ag(k)$ that are products of isogenous CM elliptic curves. To ensure both that our products are weakly isomorphic and that there are enough pairs, we first establish basic observations about primes in CM fields.


\begin{lemma}\label{lem:K-exists}
  Suppose $\ell$ is a prime number, 
  $g \in\Z_{\ge 1}$, and $K$ is an imaginary quadratic field over which $\ell$ splits into principal prime ideals, namely $\ell = \alpha\overline{\alpha}$ with $\alpha\in \sO_K$.
Then there are infinitely many rational prime numbers $q$ such that
    \begin{enumerate}
      \item $q$ is inert in $K$,
      \item $q \equiv -1 \mod{g}$, and
      \item $\alpha \mod q\sO_K \in (\sO_K/q\sO_K)^\times$ is a $g$th power.
    \end{enumerate}
\end{lemma}
\begin{proof}
  Let $L = K(\zeta_g,\alpha^{1/g},\overline{\alpha}^{1/g})$.
  Let $\sigma$ be any complex conjugation in $\Gal(L/\Q)$.
  Let $\mathfrak{q}$ be any prime of $L$ unramified in $L/\Q$ whose Frobenius $\mathrm{Frob}(\mathfrak{q}) \in \Gal(L/\Q)$ is $\sigma$. The Chebotarev density theorem guarantees that there are infinitely many such $\mathfrak{q}$.
  Let $q$ be the rational prime below $\mathfrak{q}$.
  Since $\sigma|_K$ is complex conjugation, $q$ is inert in $K$, giving (1).
  Since $\sigma|_{\Q(\zeta_g)}$ is complex conjugation, $q \equiv -1 \pmod{g}$, giving (2).
Since $\sigma$ has order $2$, we have $\sO_L/\mathfrak{q} \cong \sO_K/q\sO_K \cong \F_{q^2}$. In particular, reducing $\alpha^{1/g}$
modulo ${\mathfrak q}$ gives a $g$th root of $\alpha$
 in $(\sO_K/q\sO_K)^\times$, giving (3).
\end{proof}

We now establish some properties of CM elliptic curves that we will use in Proposition \ref{prop:lim-degree}.
\begin{definition}
  If $K$ is an imaginary quadratic field and $q$ is a prime number that is inert in $K$, let
  $$\Gt_q=(\sO_K/q\sO_K)^\times/(\Z/q\Z)^\times. $$
\end{definition}

\begin{lemma}\label{lem:c-torsor}
  Suppose $E$ is an elliptic curve over an algebraically closed field $k$, and $\End(E) \cong \sO_K$ for some imaginary quadratic field $K$. Let $q$ be a prime number that is inert in $K$ and not equal to the characteristic of $k$. Then the set of subgroups of $E$ of order $q$ is a torsor over $\Gt_q$.
\end{lemma}
\begin{proof}
  Since $E[q]$ is a module over $\sO_K/q\sO_K \cong \F_{q^2}$  of size $q^2$, we have that $E[q] \cong \F_{q^2}$ as $\F_{q^2}$-vector spaces. So $\F_{q^2}^\times/\F_q^\times \cong \Gt_q$ acts freely and transitively on the one-dimensional $\F_q$-subspaces of $E[q]$.
\end{proof}

\begin{definition}[{\cite[Sec.~2]{kani2011products}}]\label{def:ker-idl}
  Suppose $A$ is an abelian variety over a field. If $I$ is a regular left ideal of $\End(A)$ (i.e., a left ideal that contains an isogeny), let $H(I) = \cap_{\phi \in I}\ker \phi$. If $H$ is a finite subgroup scheme of $A$, let $I(H) = \{\phi \in \End(A) \colon \phi(H) = 0\}$. A finite subgroup scheme $H$ of $A$ is an \emph{ideal subgroup} of $A$ if $H = H(I)$ for some ideal $I$ of $\End(A)$. A left ideal $I$ of $\End(A)$ is a \emph{kernel ideal} if $I = I(H)$ for some finite subgroup scheme $H$ of $A$.
\end{definition}

\begin{lemma}[{\cite[Prop.~23]{kani2011products}}]\label{lemma:order-ideal-subgroup}
  Suppose $E$ is a CM elliptic curve. If $H = H(I)$ is an ideal subgroup of $E$, then $\# H = [\End(E):I]$.
\end{lemma}

\begin{theorem}[{\cite[Thm.~20b]{kani2011products}}]\label{thm:kani-20b}
  Suppose $E$ is a CM elliptic curve over an algebraically closed field and $H$ is a finite subgroup scheme of $E$. Then $H$ is an ideal 
  subgroup if and only if 
  $\End(E) \subseteq \End(E/H)$. Moreover, if $H_1$ and $H_2$ are ideal subgroups of $E$, then $E/H_1 \approx E/H_2$ if and only if $I(H_1)$ and $I(H_2)$ are isomorphic as $\End(E)$-modules.
\end{theorem}

\begin{lemma}\label{lem:c-end}
  Suppose $E$ is an elliptic curve over an algebraically closed field $k$, and $\End(E) = \sO_K$ for some imaginary quadratic field $K$. Suppose $q$ is a product of distinct prime numbers that are inert in $K$ and not equal to $\car(k)$, and suppose $\sC$ is a subgroup of $E$ of order $q$. Then $\End(E/\sC) = \Z + q\sO_K \subset \sO_K = \End(E)$. In particular, every endomorphism of $E/\sC$ is induced by an endomorphism of $E$ that takes $\sC$ to $\sC$.
\end{lemma}
\begin{proof}
  Let $q = \prod_{i=1}^n q_i$ be the prime factorization of $q$.
  Then $\sC = \sum_{i=1}^n \sC_i$ where each $\sC_i$ is a subgroup of $E$ of order $q_i$. Let $E_j = E/\sum_{i=1}^j\sC_i$.

  The isogeny $\pi: E \to E/\sC$ factors into a chain of isogenies
   \[
    E = E_0
    \xrightarrow{\pi_1}
    E_1
    \xrightarrow{\pi_2}
    E_2
    \xrightarrow{\pi_3}
    \cdots
    \xrightarrow{\pi_n}
    E_n = E/\sC.
  \]
  Let $\sO_j$ be the order $\End(E_j)$ in $K$, and let $f_j$ be its conductor, i.e., $f_j = [\sO_{K} : \sO_j]$. Since the isogeny $\pi_j$ has degree $q_j$, either $[\sO_j : \sO_{j-1}] = q_j$, or $[\sO_{j-1} : \sO_j] = q_j$, or $\sO_{j-1} = \sO_j$ \cite[Prop.~5]{kohel1996endomorphism}. Thus $f_j$ divides $\prod_{i=1}^{j}q_i$. Since the $q_i$ are distinct, $q_{j+1} \nmid f_j$.
  Since $q_j$ is inert in $K$, $\sO_K$ has no prime ideal of norm $q_j$. There is a norm-preserving bijection between prime ideals of $\sO_{j}$ that do not divide $f_{j}$ and prime ideals of $\sO_K$ that do not divide $f_{j}$ \cite[Prop.~7.20]{cox2011primes}. Thus, $\sO_{j-1}$ has no prime ideal of norm $q_j$.
By Lemma \ref{lemma:order-ideal-subgroup}, it follows that $\ker\pi_j$ is not an ideal subgroup of $E_{j-1}$. 
Now by Theorem~\ref{thm:kani-20b}, we must have $[\sO_{j-1} : \sO_j] = q_j$. Then $[\sO_K : \sO_n] = \prod_{i=1}^n q_i = q$, as required.
\end{proof}

\begin{lemma}\label{lem:c-subgps-distinct-quotients}
  Suppose $E$ is an elliptic curve over an algebraically closed field $k$, and $\End(E) \cong \sO_K$ for some imaginary quadratic field $K$ with $\sO_K^\times = \{\pm 1\}$. Suppose $q$ is a prime not equal to $\car(k)$ and inert in $K$, and $\sC$ is a subgroup of $E$ of order $q$. If $\alpha,\beta \in \sO_K$ are prime to $q$, then $E/\alpha(\sC) \approx E/\beta(\sC)$ if and only if $\alpha(\sC) = \beta(\sC)$.
\end{lemma}
\begin{proof}
 Since $\alpha(\sC) = (\alpha\overline{\beta})(\beta(\sC))$, after replacing $\beta(\sC)$ with $\sC$ and $\alpha\overline{\beta}$ with $\alpha$ it suffices to prove the claim for $\beta = 1$.

  If $\alpha(\sC) = \sC$, then $E/\alpha(\sC) \approx E/\sC$, so it suffices to show the converse.

  Let $\tilde{u}: E/\alpha(\sC) \to E/\sC$ be an isomorphism. 
If $D$ is a subgroup of $E$, let $\pi_{D}$ denote the quotient map $E \to E/D$. Then $\alpha$ induces an isogeny $\tilde{\alpha}$
such that the left-hand square of the following diagram commutes.
$$
\xymatrix{
E \ar^{\alpha}[r]  \ar^{\pi_\sC}[d] \ar@/^1.6pc/^{\alpha'}[rr] & E \ar^{\exists u}[r]  \ar^{\pi_{\alpha(\sC)}}[d] & E \ar^{\pi_\sC}[d] \\
E/\sC \ar^-{\tilde{\alpha}}[r] & E/\alpha(\sC) \ar^-{\tilde{u}}[r] & E/\sC
}
$$
The map $\tilde{u}\circ\tilde{\alpha}$ is an endomorphism of $E/\sC$. By Lemma~\ref{lem:c-end}, $\tilde{u}\circ\tilde{\alpha}$ is induced by some $\alpha' \in \sO_K$ such that $\alpha'(\sC) = \sC$. That is, $\pi_{\sC}\circ\alpha' = \tilde{u}\circ\tilde{\alpha}\circ\pi_{\sC} = \tilde{u}\circ\pi_{\alpha(\sC)}\circ\alpha$.
  Since $\alpha$ and $\alpha'$ are prime to $q = \#\sC$, we have
  \[
    \sC + \ker(\alpha')
    =
    \ker(\pi_{\sC} \circ \alpha')
    =
    \ker(\tilde{u}\circ\pi_{\alpha(\sC)}\circ\alpha)
    =
    \ker(\pi_{\alpha(\sC)}\circ\alpha)
    =
    \sC + \ker(\alpha),
  \]
  and $\ker(\alpha) = \ker(\alpha')$. Thus $\alpha' = u\alpha$ for some $u \in \sO_K^\times = \{\pm 1\}$. So $\alpha(\sC) = \pm\alpha'(\sC) = \pm\sC = \sC$, as desired.

\end{proof}

\begin{proposition}\label{prop:prod-equiv-torsor}
  Suppose $E$ is an elliptic curve over an algebraically closed field $k$, and $\End(E) \cong \sO_K$ for some imaginary quadratic field $K$ with $\sO_K^\times = \{\pm 1\}$. Suppose $q$ is a prime not equal to $\car(k)$ and inert in $K$, and $\sC$ is a subgroup of $E$ of order $q$. Suppose $\alpha_1,\dots,\alpha_g,\beta_1,\dots,\beta_g \in \sO_K$ are prime to $q$, and let $\alpha = \prod_{i=1}^g\alpha_i$ and $\beta = \prod_{i=1}^g\beta_i$. Then
  \[
    \prod_{i=1}^g E/\alpha_i(\sC) \approx \prod_{i=1}^g E/\beta_i(\sC)
    \,\Leftrightarrow\,
    \alpha \equiv \beta \text{ in } \Gt_q.
  \]
\end{proposition}
\begin{proof}
   Let $M$ be the diagonal matrix with diagonal entries $\alpha_1^{-1},\dots,\alpha_{g-1}^{-1}, \alpha_g^{-1}\alpha$. Note that $M\mod{q\sO_K} \in \Sl_g(\sO_K/q\sO_K)$. By \cite[Cor.~5.2]{ktheory1964bass}, there exists $M' \in \Sl_g(\sO_K)$ such that $M \equiv M' \mod{q\sO_K}$. The matrix $M'$ corresponds to an automorphism of $E^g$ that sends $\prod_{i=1}^g \alpha_i(\sC)$ to $\sC^{g-1} \times \alpha(\sC)$. A similar construction with $\beta$ shows that
  \[
    \prod_{i=1}^g E/\alpha_i(\sC) \approx \prod_{i=1}^g E/\beta_i(\sC)
    \,\Leftrightarrow\,
    \left(E/\sC\right)^{g-1} \times E/\alpha(\sC) \approx \left(E/\sC\right)^{g-1} \times E/\beta(\sC).
  \]

  Note that $\alpha$ induces an isogeny $\tilde{\alpha}: E/\sC \to E/\alpha(\sC)$. Let $\sO = \Z + q\sO_K$. By Lemma~\ref{lem:c-end} we have $\sO \cong \End(E/\sC)\cong \End(E/\alpha(\sC))\cong \End(E/\beta(\sC))$. Thus $\ker\tilde{\alpha}$ is an ideal subgroup by Theorem~\ref{thm:kani-20b}. The same holds for $\beta$ in place of $\alpha$.

  For an isogeny $\phi$ with domain $E/\sC$, let $I_\phi$ denote the kernel ideal of $\ker(\phi)$, i.e., $I_\phi = \{f \in \End(E/\sC) \colon f(\ker(\phi)) = 0\}$.
  By \cite[Thm.~46]{kani2011products},
  \begin{multline*}
    \left(E/\sC\right)^{g-1} \times E/\alpha(\sC) \approx \left(E/\sC\right)^{g-1} \times E/\beta(\sC)
    \\
    \Leftrightarrow
    \left(\bigoplus_{i=1}^{g-1} \sO\right) \oplus I_{\tilde{\alpha}} \cong \left(\bigoplus_{i=1}^{g-1} \sO\right) \oplus I_{\tilde{\beta}} \text{  as $\sO$-modules}.
  \end{multline*}
By \cite[Thm.~48]{kani2011products} (see also \cite[Rem.~49b]{kani2011products}), the latter isomorphism is equivalent to $I_{\tilde{\alpha}} \cong I_{\tilde{\beta}}$.
  By Theorem~\ref{thm:kani-20b},
  \[
    I_{\tilde{\alpha}} \cong I_{\tilde{\beta}}
    \,\Leftrightarrow\,
    E/\alpha(\sC) \approx E/\beta(\sC).
  \]
  By Lemma~\ref{lem:c-subgps-distinct-quotients},
  \[
    E/\alpha(\sC) \approx E/\beta(\sC)
    \,\Leftrightarrow\,
    \alpha(\sC) = \beta(\sC).
  \]
  By Lemma~\ref{lem:c-torsor}, $\alpha(\sC) = \beta(\sC)$ if and only if $\alpha \equiv \beta$ in $\Gt_q$.
\end{proof}

Recall the map $\psimod$ defined in \eqref{def:psimoddef}.

\begin{lemma}\label{lem:psimod-weakly-isomorphic-to-product}
Suppose $g\in\Z_{\ge 1}$, and $A$ is a $g \times g$ symmetric, positive definite integer matrix whose determinant $N$ is not divisible by the characteristic of our base. 
Suppose $(E,C) \in Y_0(N)$, and let $\psimod(E,C) = (B,\lambda)$. Let $d_1, \cdots, d_g$ be the elementary divisors of $A$, and for $1 \le i \le g$ let $E_i = E/(N/d_i)C$. Then $\lambda$ is a principal polarization and 
$
    B \approx E_1 \times \cdots \times E_g.
$
\end{lemma}

\begin{proof}
Since $\lambda_A = \pi^*(\lambda) = \pi^\vee \circ \lambda \circ \pi$, where $\pi^\vee$ is the dual isogeny, we have $\deg(\lambda_A) = \deg(\pi)^2 \cdot \deg(\lambda)$. Therefore $\lambda$ is a principal polarization.
  
Let $D$ be the diagonal matrix with $d_1, \cdots, d_g$ on the diagonal.
Then $D$ is the Smith normal form of $A$, and there exist $U,V \in \Gl_g(\Z)$ such that $A = UDV$.
If $M \in \Gl_g(\Z)$, let $\lambda_M$ denote the natural endomorphism of $E^g$ induced by $M$.
Then 
$\lambda_A = \lambda_U \lambda_D \lambda_V$. Since
  \begin{align*}
    \ker \lambda_D &= E[d_1] \times E[d_2] \times \cdots \times E[d_g]
  \end{align*}
  we have $(\ker \lambda_D) \cap C^g = \prod_{i=1}^g (\frac{N}{d_i}C)$.
 Since $\lambda_V(C^g) = C^g$, we now have
  \[
    \lambda_V((\ker \lambda_A) \cap C^g) = (\ker \lambda_D) \cap C^g.
  \]
  Thus,
  $$
  B = E^g/((\ker \lambda_A) \cap C^g) \approx E^g/((\ker \lambda_D) \cap C^g) = E_1 \times \cdots \times E_g.
  $$
\end{proof}

\begin{proposition}\label{prop:lim-degree}
Suppose $g\in\Z_{\ge 2}$. 
  Then there are infinitely many rational primes $\ell\neq\car(k)$ such that if
   $A,A' \in \detl$, $\det(A) = \ell^n$, and $\det(A') = \ell^m$,
  then there exists an infinite sequence $x_i \in Y_0(\ell^n)(k)$ such that
  \[
    \lim_{i \to \infty}\#\left\{ y \in Y_0(\ell^m)(k) \colon \psi_A(x_i) \approx \psi_{A'}(y) \right\} = \infty.
  \]
\end{proposition}
\begin{proof}
  Let $K$ be an imaginary quadratic field 
with $\sO_K^\times = \{\pm 1\}$ and such that if $\car(k) > 0$ then $\car(k)$ splits in $K$.
Let $\ell$ be any prime number not equal to $\car(k)$ that splits into principal primes, $\ell = \alpha{\bar \alpha}$ with $ \alpha\in \sO_K$.
There are infinitely many such primes $\ell$ by the Chebotarev density theorem.
Choose an infinite sequence $\{q_i\}$ of rational primes as in Lemma~\ref{lem:K-exists}. Since $\ell$ splits and the $q_i$ are
inert in $K$, we have $q_i\ne \ell$ for all $i$. 

  The classical theory of complex multiplication shows that there is an elliptic curve $E_0$ defined over a number field $L$ with good reduction everywhere and with $\End(E_0) \cong \sO_K$.
  If $\car(k) = 0$, then the base change $E$ of $E_0$ to $k$ via any embedding of $L$ into $k$ is an elliptic curve over $k$ with $\End(E) \cong \sO_K$.
  If $p:= \car(k) > 0$, then since $p$ splits in $K$, the curve $E_0$ has ordinary reduction at all primes above $p$;
   if $E$ is the base change of $E_0$ to $k$ via any reduction map of $\sO_L$ to $k$, then $E$ is an ordinary elliptic curve and $\End(E)$ contains a copy of $\sO_K$, so $\End(E) \cong \sO_K$. Thus, there is an elliptic curve $E$ over $k$ such that $\End(E) \cong \sO_K$.
  Fix such a curve $E$.

  If $\sC$ is a cyclic subgroup of $E$ of order prime to $\ell$, then $\alpha$ induces a chain of isogenies
  \[
    E/\sC \to E/\alpha(\sC) \to \cdots \to E/\alpha^n(\sC).
  \]
  Let $\alpha_{\sC,n}: E/\sC \to E/\alpha^n(\sC)$ be the composition of this chain. Then $\alpha_{\sC,n}$ is a cyclic $\ell^n$-isogeny, so the pair $(E/\sC,\alpha_{\sC,n})$ defines a point in $Y_0(\ell^n)(k)$.

    For each $i$, fix a cyclic subgroup $\sC_i$ of $E$ of order $q_i$. Define $x_1$ to be $(E/\sC_1,\alpha_{\sC_1,n})$, $x_2$ to be $(E/(\sC_1 + \sC_2), \alpha_{\sC_1 + \sC_2,n})$, and so on. The $x_i$ are distinct points, since the curves all have different endomorphism rings by Lemma~\ref{lem:c-end}.

Let $\ell^{n_1},\dots,\ell^{n_g}$ be the elementary divisors of $A$  and let $\ell^{n_1'},\dots,\ell^{n_g'}$ be the elementary divisors of $A'$.
By Lemma~\ref{lem:psimod-weakly-isomorphic-to-product} we have $$\psi_{A}(x_1) \approx E/\alpha^{n_1}(\sC_1) \times \cdots \times E/\alpha^{n_g}(\sC_1).$$
    Let $S$ denote a set of representatives $\beta \in \sO_K$ of the solutions to 
    \[
      \alpha^{n_1 + \cdots + n_g} \equiv \beta^g\alpha^{n_1' + \cdots + n_g'}
      \text{ in } \Gt_{q_1}.
    \]
Since $\alpha$ is a $g$-th power in  $(\sO_K/q_1\sO_K)^\times$, and  $g$ divides the order $q_1 + 1$ of the cyclic group $\Gt_{q_1}$,
   we have $\#S = g$. For all $\beta \in S$, let $y_\beta = (E/\beta(\sC_1),\alpha_{\beta(\sC_1),m})$. By Lemma~\ref{lem:psimod-weakly-isomorphic-to-product} we have $$\psi_{A'}(y_\beta) = E/\alpha^{n_1'}(\beta(\sC_1)) \times \cdots \times E/\alpha^{n_g'}(\beta(\sC_1)).$$ Thus $\psi_A(x_1) \approx \psi_{A'}(y_\beta)$ by Proposition~\ref{prop:prod-equiv-torsor}. By Lemma~\ref{lem:c-subgps-distinct-quotients} the $y_\beta$ are distinct. This gives $g$ points $y \in Y_0(\ell^m)(k)$ with $\psi_{A}(x_1) \approx \psi_{A'}(y)$.

  For $x_2$, we use a similar construction. By the Chinese remainder theorem, the set of subgroups of $E$ of order $q_1q_2$ is a torsor over $\Gt_{q_1} \times \Gt_{q_2}$. Finding $\beta \in \sO_K$ such that
  \[
    \psi_{A}(x_2) \approx \psi_{A'}\left(E/\beta(\sC_1+\sC_2),\alpha_{\beta(\sC_1+\sC_2),m}\right)
  \]
  reduces to finding solutions $\beta$ to 
  \[
    \alpha^{n_1 + \cdots + n_g} \equiv \beta^g\alpha^{n_1' + \cdots + n_g'}
    \text{ in } \Gt_{q_1} \times \Gt_{q_2}.
  \]
There are precisely $g^2$ solutions $\beta$. Continuing this construction for $i=3,4,\dots$, we find that for each $i$ there are $g^i$ points $y \in Y_0(\ell^m)(k)$ such that $\psi_{A}(x_i) \approx \psi_{A'}(y)$. Since $g > 1$, the result follows.
\end{proof}

\begin{proof}[Proof of Theorem~\ref{thm:Sg-dense}]
Suppose $\ell$ is a prime as in Proposition \ref{prop:lim-degree}, and $A,A'\in\detl$.
  Let $S$ denote the Zariski closure of $\sg(k) \cap (X_A \times X_{A'})(k)$ in $X_A \times X_{A'}$. By Proposition~\ref{prop:lim-degree}, the set $\sg(k) \cap (X_A \times X_{A'})(k)$ has an infinite number of geometric points. Thus $\dim S \geq 1$. Suppose $\dim S = 1$, so that $S$ is a finite union of curves $\cup S_i$ and isolated points. The $S_i$ cannot all be horizontal components---that is, of the form $X \times \{z\}$---since this would contradict Proposition~\ref{prop:lim-degree}. Let $S'$ be $S$ with the horizontal components and isolated points removed. By Proposition~\ref{prop:lim-degree}, the projection map $\pi_X: S' \to X$ has unbounded degree. But $\pi_X$ on each irreducible component of $S'$ is nonconstant, so $\pi_X|_{S'}$ has finite degree. This contradiction shows that $\dim S \geq 2$, and the desired result follows.
\end{proof}

\section{Proof of Theorem \ref{thm:arbit-char}}
\label{sec:characteristic-p}

\begin{lemma}\label{lemma:travis-lemma}
Suppose $p$ and $\ell$ are distinct prime numbers.  There are infinitely many
imaginary quadratic fields $K$ such that $p$ splits in $K$ and $\ell$ is inert in $K$.
\end{lemma}

\begin{proof}
If $q$ is an odd prime not dividing an integer $d$, then $q$ splits (respectively, is inert)
in $\Q(\sqrt{d})$ if $d$ is a square (respectively, is a nonsquare) modulo $q$.
Thus if $p$ and $\ell$ are both odd, the lemma is satisfied with $K = \Q(\sqrt{d})$
for any negative integer $d$ such that both:
\begin{itemize}
\item
$d \pmod{p}$ is a square in $(\Z/p\Z)^\times$, and
\item
$d \pmod{\ell}$ is a nonsquare in $(\Z/\ell\Z)^\times$.
\end{itemize}
There are infinitely many such $K$, by the Chinese remainder theorem.
If $p=2$ the argument is similar, but replacing $(\Z/p\Z)^\times$ above with
$(\Z/8\Z)^\times$.  If $\ell = 2$, replace
$(\Z/\ell\Z)^\times$ with $(\Z/8\Z)^\times$.
\end{proof}

Suppose $A$ is a $g \times g$ diagonal matrix whose diagonal entries are positive integers, and let $N = \det A$. Observe that the map $\psimod: Y_0(N) \to \ag$ factors through the map $\ao^g \to \ag$ that sends a tuple of polarized elliptic curves $(E_1,\dots,E_g)$ to the product $E_1 \times \cdots \times E_g$ with the product polarization. Write
\[
\psidiag: Y_0(N) \to \ao^g
\]
for the associated morphism. Let $d_1, \ldots, d_g$ be the diagonal entries of $A$. If $(E, C) \in Y_0(N)$, then
\[
  \psidiag(E,C) = E_1 \times \cdots \times E_g
\]
where  $E_i := E/\frac{N}{d_i}C$.
Let $$X^{(1)}_A = \psidiag(Y_0(N)).$$

  If $g$ is a positive integer and $\ell$ is a prime number, let $W_{\ell,g}$ denote the set of all $g \times g$ diagonal matrices with diagonal entries $\ell^{n_1},\dots,\ell^{n_g}$ for some positive integers $n_1,\dots,n_g$.

\begin{theorem}\label{thm:curves-dense-a1}
    If $g \in \Z_{\ge 1}$ and $\ell$ is a
    prime number not equal to the characteristic of $k$, 
    then $\bigcup_{A\in W_{\ell,g}} X^{(1)}_A$ is Zariski dense in $\ao^g$.
\end{theorem}
\begin{proof}
 If $\car(k)=0$, let $K$ be any imaginary quadratic field, and if $\car(k)> 0$, let
  $K$ be an imaginary quadratic field in which $\car(k)$ splits and $\ell$ is inert.
  Such a $K$ exists by Lemma \ref{lemma:travis-lemma}.
 Let $E_0$ be an elliptic curve over $k$ with CM by $\sO_K$.
 Fix a subgroup $\sC_0$ of $E_0$ of order $\ell$. Let $\pi_0$ denote the quotient map $E_0 \to E_0/\sC_0 =: E_1$. Since $\End(E_0)$ has no ideal of norm $\ell$, $\sC_0$ is not an  ideal subgroup of $E_0$ by Lemma~\ref{lemma:order-ideal-subgroup}. Theorem~\ref{thm:kani-20b} and~\cite[Prop.~5]{kohel1996endomorphism} now imply that $\End(E_1) = \Z + \ell\sO_K$.

 For $i \ge 1$, let $\sC_i$ be any subgroup of $E_i$ of order $\ell$ other than the kernel of the dual isogeny $\pi_{i-1}^\vee$, and let $\pi_i$ be the quotient map $E_i \to E_i/\sC_i =: E_{i+1}$. In order to show that the $E_i$ are non-isomorphic, we will show inductively that $\End(E_{i+1}) = \Z + \ell^{i+1}\sO_{K}$. Since $\End(E_{i}) \subset \End(E_{i-1})$, Theorem~\ref{thm:kani-20b} implies that $\ker \pi_{i-1}^\vee$ is an ideal subgroup. Since $\End(E_{i}) = \Z + \ell^{i}\sO_{K}$, the only ideal of index $\ell$ in $\End(E_{i})$ is $\ell\Z + \ell^i\sO_K$. From Lemma~\ref{lemma:order-ideal-subgroup} and the uniqueness of the ideal of index $\ell$, it follows that $\ker \pi_{i-1}^\vee$ is the unique ideal subgroup of order $\ell$, namely $H(\ell\Z + \ell^i\sO_K)$. The subgroup $\sC_i$ therefore cannot also be an ideal subgroup of $E_i$. By Theorem~\ref{thm:kani-20b}, $\End(E_{i+1})$ must be strictly smaller than $\End(E_{i})$. The latter observation together with \cite[Prop.~5]{kohel1996endomorphism} implies $[\End(E_i):\End(E_{i+1})] = \ell$, and hence $\End(E_{i+1}) = \Z + \ell^{i+1}\sO_{K}$.

  Let $S = \{E_i\}_{i=0}^\infty$. Let $(E_{n_1},\dots,E_{n_g}) \in S^g$ and let $A$ be the diagonal matrix with diagonal entries $\ell^{n_1},\dots,\ell^{n_g}$.  Choose any $n > \max\{n_i\}$, and let $\sC = \ker \pi^{(n)}$. Then $(E_0,\pi^{(n)}) \in Y_0(\ell^n)(k)$ and
 \[
   \psidiag(E_0,\sC) = E_{n_1} \times \cdots \times E_{n_g}.
 \]
Thus, $S \subset \bigcup_A X^{(1)}_A(k)$. Since $S$ is an infinite set of non-isomorphic elliptic curves, $S$ is dense in $\ao$. Therefore $S^g$ is dense in $\ao^g$.
\end{proof}

\begin{theorem}\label{thm:Sg-dense-a1}
  Suppose $g\in\Z_{\ge 2}$ and $k$ is an algebraically closed field. Then there are infinitely many primes $\ell$ such that for all diagonal matrices $A$ and $A'$ in $\detl$, $\sg^{(1)}(k) \cap (X^{(1)}_A \times X^{(1)}_{A'})(k)$ is Zariski dense in $X^{(1)}_A \times X^{(1)}_{A'}$.
\end{theorem}

\begin{proof}
  The proof is the same as the proof of Theorem~\ref{thm:Sg-dense}, only replacing $\ag$ with $\ao^g$, and $\psimod$ with $\psidiag$.
\end{proof}

\begin{proof}[Proof of Theorem~\ref{thm:arbit-char}]
  The proof of Theorem~\ref{thm:arbit-char} is almost the same as that of Theorem~\ref{thm:invariant-c-constant}, replacing $\ag$  with $\ao^g$, and invoking Theorems~\ref{thm:Sg-dense-a1} and  Theorem~\ref{thm:curves-dense-a1} in place of Theorems~\ref{thm:Sg-dense} and \ref{thm:curves-dense}, respectively.
\end{proof}

\section{Proof of Theorem~\ref{thm:curves-dense}}
\label{sec:step-1}

To prove Theorem~\ref{thm:curves-dense}, we will show that for each non-zero Siegel modular function $f$ on $\ag$, there exists $A \in \detl$ such that the pullback modular function $\psimod^*(f)$ is non-zero.
Proposition \ref{prop:unique-minimizer} and Lemma \ref{lem:rggt-dense} will allow us to choose such a matrix $A$. For background on Siegel modular functions over an arbitrary field, needed for Proposition~\ref{prop:q-expansion}, see~\cite[\S{}V.1]{faltings1990degeneration}.

\begin{lemma}\label{lem:sl-dense}
  If $g\in\Z_{\ge 1}$ and $N\in\Z_{\ge 2}$, then $\Sl_g(\Z[1/N])$ is dense in $\Sl_g(\R)$.
\end{lemma}
\begin{proof}
  Let $G \in \Sl_g(\R)$. Factor $G$ as a product of elementary matrices $G = E_n \cdots E_2 E_1$, where $\det E_i = \pm 1$. For each $E_i$, with at most one exception the entries are $0$ or $\pm 1$. Thus there exists a $g \times g$ matrix $E_i'$ with entries in $\Z[1/N]$ and $\det E_i' = \det E_i$ that is arbitrarily close to $E_i$. Let $G' = E_n' \cdots E_1' \in \Sl_g(\Z[1/N])$. Since matrix multiplication is continuous, $G'$ can be made arbitrarily close to $G$.
\end{proof}

\begin{lemma}\label{lem:rggt-dense}
  If $g\in\Z_{\ge 1}$ and $N\in\Z_{\ge 2}$, then the set
\[
\{ rGG^t \colon r \in \R^+ \text{ and } G \in \Sl_g(\Z[1/N]) \}
\]
is dense in the set of $g \times g$ real, symmetric, positive definite matrices.
\end{lemma}
\begin{proof}
  Let $A$ be a $g \times g$ real, symmetric, positive definite matrix. Then there exists an orthogonal matrix $O$ such that $D = OAO^t$ is diagonal and has positive entries. Let $H = \det(D)^{-1/2}O\sqrt{D}$ so that $A = \det(D)HH^t$. Note that $H \in \Sl_g(\R)$, so by Lemma~\ref{lem:sl-dense}, there exists $G \in \Sl_g(\Z[1/N])$ arbitrarily close to $H$, and therefore $\det(D)GG^t$ can be made arbitrarily close to $A$.
\end{proof}

\begin{definition}
  A $g \times g$ matrix $Q$ is \emph{half-integral} if $2Q_{ij} \in \Z$ and $Q_{ii} \in \Z$ for all $1 \leq i,j \leq g$, where $Q_{ij}$ denotes the $i,j$ entry of $Q$.
\end{definition}

\begin{lemma}\label{lem:finite-fixed-trace}
 Fix $t\in\R$. Then there are finitely many $g \times g$ symmetric, positive semi-definite, half-integral matrices $Q$ such that $\tr(Q) \leq t$.
\end{lemma}
\begin{proof}
  Suppose $Q$ is a symmetric, positive semi-definite, half-integral matrix with $\tr(Q) \leq t$. Let $\lambda_1,\dots,\lambda_g$ be the eigenvalues of $Q$. Since $\lambda_i \geq 0$, we have
  \[
    \sum_{1 \leq i,j \leq g} Q_{ij}^2 = \tr(Q^2) = \sum_{i=1}^g \lambda_i^2 \leq \left(\sum_{i=1}^g \lambda_i\right)^2 = \tr(Q)^2 \leq t^2.
  \]
  This shows that $Q$, as a vector in $\R^{g^2}$, lies in the ball of radius $t$. Thus the number of possible $Q$ is bounded by the number of half-integral points in the ball of radius $t$ in $\R^{g^2}$, which is finite.
\end{proof}


\begin{proposition}\label{prop:unique-minimizer}
  Suppose $g\in\Z_{\ge 1}$ and $S$ is a nonempty set of $g \times g$ symmetric, positive semi-definite, half-integral matrices. Then there exists an open cone of $g \times g$, positive definite, real, symmetric matrices $A$ such that $\displaystyle\min_{Q \in S}\tr(AQ)$ is achieved by a unique $Q\in S$.
\end{proposition}
\begin{proof}
Let $t_0 = \min_{Q \in S}\tr(Q)$.
  Since $\{Q \in S \colon \tr(Q) \leq t\}$ is finite for all $t \in \R$ by Lemma~\ref{lem:finite-fixed-trace},
  there is a gap between $t_0$ and the next smallest value $t_1$ of $\tr(Q)$ for $Q \in S$. Let $S_0 = \{Q \in S \colon \tr(Q) = t_0\}$.

  Let $\mathfrak{S}$ be the set of real symmetric $g \times g$ matrices $E$ such that
  \begin{enumerate}
    \item $\tr(EQ_i) \neq \tr(EQ_j)$ for all distinct $Q_i,Q_j \in S_0$,
    \item $E$ is positive definite, and
    \item $\tr(EQ) < t_1 - t_0$ for all $Q\in S_0$.
  \end{enumerate}
  To see that $\mathfrak{S}$ is non-empty, observe that (1) describes the complement of a finite union of hyperplanes, which is a dense set. Therefore there is a real symmetric matrix $E$ satisfying both (1) and (2). Scaling by a sufficiently small $\varepsilon > 0$, we can guarantee (3) as well.

  Let $C = \R^+(I + \mathfrak{S})$, where $I$ is the $g \times g$ identity matrix. By definition, $C$ is a cone. The set $C$ is open because $\mathfrak{S}$ is defined as a finite intersection of open subsets.

  Suppose $A \in C$; that is, $A = r(I+E)$ for some $r \in \R^+$ and $E \in \mathfrak{S}$. Let $Q \in S$. Since $E$ and $Q$ are positive semi-definite, $\tr(EQ)\ge 0$. Property (3) now implies that $\tr((I+E)Q)$ is minimized only when $Q \in S_0$. Hence by (1) we have that $\tr((I+E)Q)$ is minimized for a unique $Q \in S$. Since $E$ is positive definite, so is $I+E$, and hence $I + E$ satisfies the conclusion of the proposition. By linearity of the trace, $A$ also satisfies the conclusion of the proposition, and the desired result follows.
\end{proof}

Let $\hh$ denote the complex upper half plane and let $\hh_g$ denote the degree $g$ Siegel upper half space. Recall that $\Sp_{2g}(\Z)$ (resp., $\Gamma_0(N)$) acts on $\hh_g$ (resp., $\hh$) by fractional linear transformations,
 $Y_0(N) = \hh/\Gamma_0(N)$, and $\ag = \hh_g/\Sp_{2g}(\Z)$.
\begin{lemma}\label{lem:psi-A-n}
Suppose that $A$ is a $g \times g$ symmetric, positive definite integer matrix and $N = \det A$.
Then the map \[
  \hh \to \hh_g, \quad  \tau \mapsto \tau A
\]
induces a morphism of $\C$-schemes
$
  \psimodt: Y_0(N) \longrightarrow \ag.
$
\end{lemma}

\begin{proof}
Let
\[
\sigma = \begin{bmatrix} a & b \\ c & d \end{bmatrix} \in \Gamma_0(N) \text{ and } M = \begin{bmatrix} aI_g & bA \\ cA^{-1} & dI_g \end{bmatrix}
\]
where $I_g$ denotes the $g \times g$ identity matrix.
Suppose $\tau \in \hh$. Then $\tau A$ is symmetric and its imaginary part is positive definite, so $\tau A \in \hh_g$.
A direct computation shows that $M \in\Sp_{2g}(\Z)$
and  $\sigma(\tau)A = M(\tau A)$.
\end{proof}

Recall the definition of $\psimod$ in \eqref{def:psimoddef}.

\begin{proposition}\label{prop:psimod-over-c}
Suppose that $g\in\Z_{\ge 1}$, $\ell$ is a prime, $A \in \detl$, and
 the base scheme is $\C$. Then $\psimod = \psimodt$.
\end{proposition}

\begin{proof}
Let  $N = \det A$.  Suppose $(E, C) \in Y_0(N)$. There exists $\tau \in \hh$ such that $E = \C/(\Z + \tau \Z)$ and $C = \langle 1/N \rangle$. We first show that $\psimodt(E,C) \approx \psimod(E,C)$.

  Let $\Lambda = \Z^g + \tau \Z^g \subset \C^g$. Then $\C^g/\Lambda = E^g$.
  Let
  \[
    \tla = \Z^g + \tau A \Z^g \quad \text{ and } \quad     \Lambda_A = A^{-1}\Z^g + \tau \Z^g.
  \]
By definition, $\psimodt(E,C) \approx \C^g/\tla$.
 Since $\Lambda \subset \Lambda_A$, there is a natural isogeny
  \[
    \rho: E^g = \C^g/\Lambda \to \C^g/\Lambda_A.
  \]

  Multiplication by $A: \C^g \to \C^g$ induces an isomorphism
  \[
     \C^g/\Lambda_A \xrightarrow{\sim}  \C^g/\tla
  \]
and an isogeny
  \[
    \lambda_A: \C^g/\Lambda \longrightarrow \C^g/\Lambda,
  \]
  where $\lambda_A$ is the natural isogeny $E^g \to E^g$ induced by the matrix $A$. Then
\[
\ker \lambda_A = A^{-1}\Z^g + \tau A^{-1} \Z^g \pmod{\Lambda}
\]
and
\[
\ker \rho = (\ker \lambda_A) \cap C^g. 
\]
Now 
\[
  \psimod(E,C) \approx E^g/((\ker \lambda_A) \cap C^g) = E^g/(\ker \rho)
  \cong \C^g/\Lambda_A
  \cong \C^g/\tla
  \approx \psimodt(E,C).
\]

  Elements of $\ag$ can be viewed as pairs $(\C^g/(\Z^g + \Omega \Z^g,\sE)$ where
  $\Omega \in \hh_g$ and where $\sE: \C^g \times \C^g \to \R$ is the alternating Riemann form that
 satisfies $\sE(u,\Omega v) = u^tv$ for all $u,v \in \R^g$.
See for example \cite{rosen-avc}.

The (non-principal) polarization $\lambda_A$ on $E^g = \C^g/\Lambda$ corresponds to the alternating Riemann form that satisfies $\sE_A(u,\tau v) = u^tAv$ for all $u,v \in \R^g$. Since $\sE_A(\Lambda_A \times \Lambda_A) \subseteq \Z$, it follows that $\sE_A$ is an alternating Riemann form for $\C^g/\Lambda_A$ as well.
Since the polarization $\sE_A$ descends from $\lambda_A$, and the polarization coming from $\psimod$ is the unique polarization also descending from $\lambda_A$, we have
\[
  \psimod(E,C) = (\C^g/\Lambda_A, \sE_A).
\]
We have $\psimodt(E,C) = (\C^g/\tla, \tilde{\sE})$ where
 $\tilde{\sE}$ denotes the 
alternating Riemann form 
that satisfies 
$\tilde{\sE}(u, \tau A v) = u^tv$ for all $u,v \in \R^g$. If $u, v \in \R^g$, then
$$  \tilde{\sE}(u, \tau A v) = u^tv
                          = \sE_A(A^{-1}u, \tau v)
                          = \sE_{A}(A^{-1}u, A^{-1}\tau Av).$$
Thus multiplication by $A$ on $\C^g$ induces an isomorphism of polarized abelian varieties
\[
  \psimod(E,C) = (\C^g/\Lambda_A, \sE_A) \xrightarrow{\sim} (\C^g/\tla, \tilde{\sE}) = \psimodt(E,C),
\]
as desired.
\end{proof}

Suppose $g$ is a positive integer.
   Let $\sD$ denote the set of $g \times g$ symmetric, half-integral, positive semidefinite matrices.
   If $f$ is a Siegel modular form on $\ag$, $\ell$ is a prime not equal to the characteristic of the base, and $A\in \detl$, let $\psimod^*(f) = f \circ \psimod$ denote the pullback modular form on $X_0(\det(A))$.
     Note that the cusp $\infty$ on $X_0(\det(A))$ is characterized by the fact that the universal elliptic curve has type $I_1$ reduction there.

\begin{proposition}\label{prop:q-expansion}
  Suppose  $g \in \Z_{\ge 1}$, $\ell$ is a prime number,
  $k$ an algebraically closed field with $\car(k) \neq \ell$, 
  and $f$ a $k$-valued Siegel modular form on $\ag$.
  Then there is a function  $c : \sD \to k$ such that for all 
  $A \in \detl$, the $q$-expansion of
  $\psimod^*(f)$ at $\infty$ is
  \begin{equation}
    \sum_{Q \in \sD} c(Q) q^{\tr(AQ)}.\label{eq:q}
  \end{equation}
\end{proposition}

\begin{proof}
  We first prove the result over $\C$. Let $\overline{\ag}$ be a toroidal compactification of $\ag$. Let $A \in \detl$ and let $N=\det(A)$. The rational map $\psimod$ extends to a morphism $X_0(N) \to \overline{\ag}$. Let $\tau_{ij}$ (resp.~$\tau$) with $1 \leq i \leq j \leq g$ be the standard coordinates for $\hh_g$ (resp.~$\hh$), and let $q_{ij}$ (resp.~$q$) be $e^{2\pi i \tau_{ij}}$ (resp.~$e^{2\pi i \tau}$). By Proposition~\ref{prop:psimod-over-c}, $\psimod^*(q_{ij}) = q^{A_{ij}}$. Setting $q_{ji} = q_{ij}$, the modular form $f$ has an expansion of the form
    \[
      \sum_{Q \in \sD} c(Q) \prod_{i,j=1}^g q_{ij}^{Q_{ij}}.
    \]
    Thus at the cusp $\infty$ of $X_0(N)$, the $q$-expansion of $\psimod^*(f)$ is
    $$\sum_{Q \in \sD} c(Q) \prod_{i,j=1}^g q^{A_{ij}Q_{ij}} = \sum_{Q \in \sD} c(Q) q^{\sum_{i,j=1}^g A_{ij}Q_{ij}}
    = \sum_{Q \in \sD} c(Q) q^{\tr(AQ)}.$$

    Let $\agm$ be the minimal compactification of $\ag$. There is a contraction map $\overline{\ag} \to \agm$. Composing $\psimod$ with this contraction, we obtain a morphism $X_0(N) \to \agm$, which we will also call $\psimod$. By~\cite[Theorem V.2.3]{faltings1990degeneration} we have 
    $$\agm = \textrm{Proj}(\oplus_{j=0}^\infty \Gamma(\overline{\ag}, \omega^j)),$$ 
    where $\omega^j$ is the sheaf of Siegel modular forms of weight $j$. Since 
    the sheaf of weight $2$ modular forms is ample, we have 
    $$X_0(N) = \textrm{Proj}(\oplus_{j=0}^\infty \Gamma(X_0(N), \omega_0^j)),$$ 
    where $\omega_0^j$ is the sheaf of modular forms of weight $j$ on $X_0(N)$. Therefore the morphism $\psimod: X_0(N) \to \agm$ is characterized by the pullback map
    \[
      \psimod^*: \bigoplus_{j=0}^\infty \Gamma(\overline{\ag}, \omega^j) \to \bigoplus_{j=0}^\infty \Gamma(X_0(N), \omega_0^j).
    \]
    Both graded algebras are finitely generated over $\Z[1/\ell]$, and furthermore $\psimod^\ast$ is Galois equivariant. Hence $\psimod$ descends to a morphism of schemes over $\Z[1/\ell]$. By functoriality of $q$-expansions, $\psimod^*$ over $\Z[1/\ell]$ is given by formula~\eqref{eq:q}. Now base-extend to $k$. Let $f$ be a Siegel modular form on $\ag$. By~\cite[Theorem V.2.3]{faltings1990degeneration}, $\agm$ has only one zero-dimensional cusp, so $f$ has a unique $q$-expansion, say $\sum_{Q \in \sD} c(Q) \prod_{i,j=1}^g q_{ij}^{Q_{ij}}$. Since $\psimod^*$ over $k$ is obtained by base-extension from $\Z[1/\ell]$, formula~\eqref{eq:q} holds over $k$ as well.
\end{proof}

\begin{proof}[Proof of Theorem~\ref{thm:curves-dense}]
  We claim that $\cup X_A$ is Zariski dense in $\ag$. The claim is trivial when $g = 1$, so we assume $g \geq 2$. It suffices to show that for all non-zero Siegel modular functions $f: \ag \to k$, there exists $A \in \detl$ such that the pullback $\psimod^*(f)$ is non-zero. View $f$ as a global section of the structure sheaf. Then by Proposition~\ref{prop:q-expansion}, there exist coefficients $c(Q) \in k$ for $Q \in \sD$, such that for almost all $A \in \detl$, $\psimod^*(f)$ admits a $q$-expansion $\sum_{Q \in \sD} c(Q) q^{\tr(AQ)}$.

  Let $S = \{ Q \in \sD :c(Q) \neq 0\}$. By Proposition~\ref{prop:unique-minimizer} applied to the set $S$, there is an open cone $C$ of $g \times g$ real, symmetric, positive definite matrices $A$ such that $\tr(AQ)$ is minimized by a unique $Q \in S$.

  By Lemma~\ref{lem:rggt-dense}, there exists an element of $C$ of the form $rGG^t$ with $r \in \R^+$ and $G \in \Sl_g(\Z[1/\ell])$. Since $C$ is a cone, we may scale this element to obtain a matrix $A \in \detl \cap C$. 

  Let $n = \min \{\tr(AQ) \colon Q \in S\}$. By the definition of $C$, there is a unique $Q_0 \in S$ such that $\tr(AQ_0) = n$.
    Thus, the coefficient of $q^n$ in the $q$-expansion of $\psimod^*(f)$ is $c(Q_0)$. Since $Q_0 \in S$, we have $c(Q_0) \ne 0$. Hence $\psimod^*(f) \neq 0$, as desired.
\end{proof}

\begin{remark}
  In the case $k=\C$, we can give a simpler proof of Theorem~\ref{thm:curves-dense}, as follows. 
  By Lemma~\ref{lem:psi-A-n} and Proposition~\ref{prop:psimod-over-c}, if $A \in \detl$ then 
the map $\hh \to \hh_g$ defined by $\tau \mapsto \tau A$
induces the map
$\psimod: Y_0(\det(A)) \to \ag$.     
Let $$V = \{\tau A :  \tau\in\hh, A \in \detl \} \subseteq \hh_g.$$
By Lemma~\ref{lem:rggt-dense}, the topological closure of $V$ in $\hh_g$ contains the purely imaginary locus, i.e., the set of matrices of the form $iB$ where $B$ is a real, symmetric, positive definite matrix. But any holomorphic function that vanishes on the purely imaginary locus must be the zero map. It follows that the image of $V$ in $\ag$ is Zariski dense, as desired.
\end{remark}

\section{Proof of Corollary~\ref{cor:agpol}}
\label{sec:Cor15proof}

Suppose that $g$ and $n$ are positive integers.
Let $\agn$ denote the moduli space for triples $(A,\lambda,L)$, where $A$ is a $g$-dimensional abelian variety, $\lambda$ is a principal polarization on $A$, and $L$ is a level $n$ structure, i.e., $L$ is a symplectic basis $(P_1,\ldots,P_g,Q_1,\ldots,Q_g)$ for the $n$-torsion $A[n]$.

 Define an equivalence relation $\medium$ on $\agn$ by $(A_1, \lambda_1, L_1) \medium (A_2, \lambda_2, L_2)$ if and only if $A_1 \weak A_2$ and there is an isomorphism $\phi: A_1 \to A_2$ such that $\phi(L_1) = L_2$. 
 Define a set $\mg(k)  \subseteq (\agn \times \agn)(k)$ by
\[
\mg(k) =  \{(x_1,x_2) \in \agn(k) \times \agn(k) \mid x_1 \medium x_2 \}. 
\]
\begin{proposition}\label{prop:medium-isomorphic}
  If $g\in\Z_{\ge 2}$, 
  $k$ is an algebraically closed field, and $n$ is an integer with $\car(k) \nmid n$, 
  then the set $\mg(k)$
  is Zariski dense in the $k$-scheme $\agn \times \agn$.
\end{proposition}

\begin{proof}
  Let $\pi: \agn \to \ag$ be the forgetful map $(A,\lambda,L) \mapsto (A,\lambda)$. Then $\pi$ is finite-to-one, and $(\pi\times\pi)(\mg(k)) = \sg(k)$. By Theorem~\ref{thm:sg-c-dense}, $\sg(k)$ is Zariski dense in $\ag \times \ag$, so the dimension of the Zariski closure of $\sg(k)$ is $2 \dim \ag$. Therefore the dimension of the Zariski closure of $\mg(k)$ is also $2 \dim \ag = 2 \dim \agn$. Since $\agn$ is connected, so is $\agn \times \agn$, and the desired result follows.
\end{proof}

If $D$ is a polarization type, define a set $\sgd(k)  \subseteq (\agd \times \agd)(k)$ by
\[
\sgd(k) =  \{((B_1,\mu_1), (B_2, \mu_2)) \in \agd(k) \times \agd(k) \mid B_1 \weak B_2 \}.
\]

\begin{proposition}
\label{prop:whatever}
  If $g\in\Z_{\ge 2}$, 
  $k$ is an algebraically closed field, $D=(d_1,\ldots,d_g)$ is a polarization type, and $\car(k) \nmid d_g$, 
  then the set $\sgd(k)$
is Zariski dense in the $k$-scheme $\agd \times \agd$.
\end{proposition}

\begin{proof}
Let $n = d_g$.
Let $\pi_D: \agn \to \agd$ be the morphism on moduli spaces induced by the morphism of stacks $f$ defined on the bottom of p.~492 of~\cite{dejong}.
Let $\xi$ be a geometric generic point for the Zariski closure of $(\pi_D \times \pi_D)(\mg(k))$ in $\agd \times \agd$.
Let $\xi'$ be a geometric generic point for $(\pi_D\times\pi_D)^{-1}(\overline{\xi})$, where $\overline{\xi}$ is the
Zariski closure of $\xi$.
Then $\mg(k) \subseteq (\pi_D\times\pi_D)^{-1}(\overline{\xi}) = \overline{\xi'}$.
By Proposition~\ref{prop:medium-isomorphic}, the Zariski closure of $\mg(k)$ is $\agn \times \agn$.
It follows that $\overline{\xi'} = \agn \times \agn$.
Since $\pi_D\times\pi_D$ is finite, we have
$$\dim(\overline{\xi}) = \dim(\overline{\xi'}) = \dim(\agn \times \agn) = \dim(\agd \times \agd).
$$
By~\cite[Prop. 1.11]{dejong}, $\agd$ is irreducible.
It follows that $\agd \times \agd = \overline{\xi}$, which is
 the Zariski closure of $(\pi_D \times \pi_D)(\mg(k))$.
 Thus, $(\pi_D \times \pi_D)(\mg(k))$ is Zariski dense in $\agd \times \agd$.

Suppose $(A_1, \lambda_1, L_1) \medium (A_2, \lambda_2, L_2)$ in $\agn$. Let $(B_i, \mu_i) = \pi_D(A_i, \lambda_i, L_i)$. From the definition of $\pi_D$, we have $B_i^\vee = A_i^\vee/H_i$ for some subgroup $H_i \subset A^\vee_i[n]$. By the definition of $\medium$, there is a weak isomorphism $A_2 \xrightarrow{\sim} A_1$ whose dual sends $H_1$ isomorphically onto $H_2$, and hence 
  $B_1^\vee \weak B_2^\vee$. It follows that $B_1 \weak B_2$, and therefore
    $(\pi_D\times\pi_D)(\mg(k)) \subseteq \sgd(k).$
The desired result now follows.
\end{proof}
The proof of Corollary~\ref{cor:agpol} is now the same as the proof of Theorem~\ref{thm:invariant-c-constant}, replacing the use of Theorem \ref{thm:sg-c-dense} with the use of Proposition \ref{prop:whatever}.

\bibliographystyle{halpha}
\bibliography{./references}

\end{document}